\input ./style/arxiv-vmsta.cfg
\documentclass[numbers,compress,v1.0.1]{vmsta}

\usepackage{dsfont}

\volume{3}
\issue{3}
\pubyear{2016}
\firstpage{223}
\lastpage{235}
\doi{10.15559/16-VMSTA63}% Updated by VTEXPTS2LaTeX.exe, 03.11.2016
\startlocaldefs

\newcommand{\rrvert}{\vert}
\newcommand{\llvert}{\vert}
\allowdisplaybreaks

\newcommand{\1}{\mathds{1}}
\newcommand{\almsur}{~\text{a.s.}}
\newcommand{\de}{\,d}
\newcommand{\E}{\mathrm{E}}

\renewcommand{\ge}{\geqslant}

\newcommand{\Law}{\mathbb{P}}
\renewcommand{\le}{\leqslant}
\newcommand{\mbR}{{\mathbb R}}
\newcommand{\mbN}{{\mathbb N}}

\newcommand{\mx}{\vee}

\newcommand{\Pb}{\mathrm{P}}

\newcommand{\R}{\mathbb{R}}
\newcommand{\sign}{\mathop{\mathrm{sign}}}
\renewcommand{\tanh}{\mathop{\mathrm{th}}}
\newcommand{\tmpdelta}{\delta}
\newcommand{\tmpnu}{\nu}
\newcommand{\ve}{\varepsilon}
\newcommand{\vf}{\varphi}

\newcommand{\wta}{\widetilde{a}}

\newcommand{\xin}[1][n]{\xi^{(#1)}}
\newcommand{\xn}[1][n]{X_{#1}}

\theoremstyle{plain}\newtheorem{theorem}{Theorem}

\newtheorem{lemma}{Lemma}
\theoremstyle{definition}\newtheorem{definition}{Definition}
\theoremstyle{remark}\newtheorem{remark}{Remark}
\endlocaldefs

\begin{document}
\begin{frontmatter}

\title{A limit theorem for singular stochastic differential~equations}
%\author[]{\inits{}\fnm{}\snm{}\corref{cor1}}\email{}
%\cortext[cor1]{Corresponding author.}

%\author[]{\inits{}\fnm{}\snm{}}\email{}

%\fnref{f1}
%\fntext[]{Some remarks}

%\address[]{}
%\address[]{}

%\markboth{A. Authors}{Title}

\author[a,b]{\inits{A.}\fnm{Andrey}\snm{Pilipenko}\corref{cor1}}\email{apilip@imath.kiev.ua}
\cortext[cor1]{Corresponding author.}
\address[a]{Institute of Mathematics of NAS of Ukraine,
Department of Stochastic Processes,
Tereschenkivska st., 3, 01601 Kyiv, Ukraine}

\author[b]{\inits{Yu.}\fnm{Yuriy}\snm{Prykhodko}}\email{prykhodko@matan.kpi.ua}
\address[b]{National Technical University of Ukraine <<KPI>>,
Department of Mathematical Analysis and Probability Theory,
Peremohy ave., 37, 03056 Kyiv, Ukraine}

%\author[]{ \fnref{f1}\corref{c1}}\email{}

\markboth{A. Pilipenko, Yu. Prykhodko}{A limit theorem for
singular stochastic differential equations}

\begin{abstract}
We study the weak limits of solutions to SDEs
\[
dX_n(t)=a_n\bigl(X_n(t)\bigr)\,dt+dW(t),
\]
where the sequence $\{a_n\}$ converges in some sense to
$(c_- \1_{x<0}+c_+ \1_{x>0})/x+\gamma\delta_0$. Here $\delta_0$ is the Dirac delta
function concentrated at zero. A limit of $\{X_n\}$ may be a Bessel
process, a skew Bessel process, or a mixture of Bessel processes.
\end{abstract}

%\begin{keyword} . \sep.
%\MSC[2010] . \sep.
%\end{keyword}


\received{19 September 2016}% Updated by VTEXPTS2LaTeX.exe, 03.11.2016
%08:24
%
\revised{23 October 2016}% Updated by VTEXPTS2LaTeX.exe, 03.11.2016
%08:24
%
\accepted{23 October 2016}% Updated by VTEXPTS2LaTeX.exe, 03.11.2016
%08:24
\publishedonline{8 November 2016}
\end{frontmatter}

\section{Introduction}

Consider the stochastic differential equation
\begin{equation}
\label{eq:main-sde} \de\xn[](t) = a\bigl(\xn[](t)\bigr)\de t + \de W(t),\quad t\ge0,
\end{equation}
where $a$ is a locally integrable function.

The aim of this paper is to study convergence in distribution of the
sequence of processes
$\{\xn[](nt)/\sqrt{n},\ t\ge0\}$ as $n\to\infty$.\vadjust{\eject}

Observe that
\[
\de\xn(t) = \sqrt{n}a\bigl(\sqrt{n}\xn(t)\bigr)\de t + \de W_n(t),\quad
t\ge0,
\]
where $W_n(t)=W(nt)/\sqrt{n}$, $t\ge0$ is a Wiener process, %
and $\xn(t)=\xn[](nt)/\sqrt{n}$, $t\ge0$%
.

Hence, to study the sequence $\{\xn[](nt)/\sqrt{n}\}$,
it suffices to investigate the SDEs
\[
\de\xn(t) = a_n\bigl(\xn(t)\bigr)\de t + \de W(t),\quad t\ge0,
\]
where $a_n(x)=na(nx)$.

If $a\in L_1(\R)$, then $a_n$ converges in generalized sense to $\alpha
\delta_0$, where $\delta_0$ is the Dirac delta function at zero,
where $\alpha=\int_\R a(x)\,dx$. It is well known that in this case the sequence
$\{\xn\}$ converges weakly to a skew Brownian motion
with parameter
$\gamma= \tanh(\alpha) = \frac{e^\alpha- e^{-\alpha}}{e^\alpha+
e^{-\alpha}}$; see, for example, \cite{Portenko1976,Lejay2006}. Recall
that \cite{HarrisonShepp1981,Lejay2006} the skew Brownian motion
$W_\gamma(t)$
with parameter $\gamma$, $|\gamma|\le1$, is a unique (strong) solution
to the SDE
\[
\de W_\gamma(t) = \de W(t) + \gamma\de L_{W_\gamma}^0(t),
\]
where $L_{W_\gamma}^0(t)=\lim_{\ve\to0+}(2\ve)^{-1}\int_0^t\1_{|W_\gamma
(s)|\le\ve}\de s$ is the local time of the process $W_\gamma$ at 0.
The process $W_\gamma$ is a continuous
Markov process with transition probability density function
$
p_{t}(x,y) = \varphi_t (x{-}y) + \gamma\sign(y)\,\varphi_t
(|x|{+}|y|)$, $x,y\in\mathbb{R}$,
where $\varphi_t(x) = \frac{1}{\sqrt{2\pi t}} e^{-x^2/2t}$ is the
density of the normal distribution $N(0,t)$. Note also that $W_\gamma$
can be obtained from excursions of a Wiener process pointing them
(independently of each other) up and down with probabilities
$p=(1+\gamma)/2$ and $q=(1-\gamma)/2$, respectively.

Kulinich et al. \cite{KulinichKaskun1997,KulinichKushnirenkoMishura2014} considered limit theorems in the case
where $a$ is nonintegrable function such that
\begin{equation}
\label{eqKulKas} \lim_{x\to\pm\infty} \frac{1}{x} \int
_{0}^{x} \big|v a(v)-c_\pm\big|\de v = 0,\quad\big|x
a(x)\big|\le C,
\end{equation}
where $c_\pm> -1/2$ are constants. In this case, $a_n(x)$ converges in
some sense to $c_-\1_{x<0}+c_+\1_{x\ge0}$ as $n\to\infty$.

For instance, if
$
a(x) = c_\pm/ x \text{ for } {\pm x} > x_0$,
then, for
$c_-<1/2<c_+$,
the sequence~$\xn$
converges weakly to a Bessel process.
If $c_-=c_+>-1/2$, then $|\xn|$ also converges weakly to a Bessel process.
The problem of weak convergence of~$\xn$ for (e.g.) $c_-=c_+>-1/2$ or
$c_-<c_+\le1/2$ was not considered.

In this paper, we generalize the results of \cite
{Portenko1976,KulinichKaskun1997} to the case
\[
a(x) = \wta(x) + \frac{\bar{c}(x)}{x}, \quad x\in\R,
\]
where
$\wta$ is %bounded and
integrable on $(-\infty;\infty)$,
and
\[
\bar{c}(x) = c_+\cdot\1_{x>1} + c_-\cdot\1_{x<-1}, \quad x\in\R.
\]

We consider all possible limit processes (depending on $c_+$ and $c_-$).
In particular, we show that, for
$c_+ = c_- < 1/2$,
the limit process is a skew Bessel process
(see Section~\ref{sec:sbp}).

\section{Bessel process. Skew Bessel process. Definition, properties}
\label{sec:sbp}

We recall the definition and some properties of Bessel processes.

Let $\tmpdelta\ge0$ and $x_0\in\R$. Consider the SDE
\begin{equation}
\label{eq:def-bessel-sq} Z\bigl(x_0^2,t\bigr) = x_0^2
+ 2\int_0^t \sqrt{\big|Z
\bigl(x_0^2,s\bigr)\big|} \de W(s) + \tmpdelta t,\quad t\ge0,
\end{equation}
where $W$ is a Wiener process.

It is known (see \cite{RevuzYor1999}, XI.1, (1.1)),
that
there exists a unique strong solution
$Z(x_0^2,\cdot)$ of~\eqref{eq:def-bessel-sq}. This solution is called
the squared
$\tmpdelta$-dimensional Bessel process. The process $Z(x_0^2,\cdot)$
is nonnegative a.s.

\begin{definition}
The process $B_{c}(x_0,t) = \sqrt{Z(x_0^2,t)}$ with $x_0\ge0$ is called
the (nonnegative) \emph{Bessel process}
with parameter $c=({\tmpdelta{-}1})/{2}$.

We will call the process $ B^-_{c}(x_0,t)= -B_{c}(x_0,t) = -\sqrt
{Z(x_0^2,t)}$ with $x_0\le0$ the nonpositive Bessel process.
\end{definition}

Recall the following properties of the Bessel process (see \cite
[Chap.~XI]{RevuzYor1999}).

The Bessel process $\xi(t) = B_c(x_0,t)$ satisfies the SDE
\[
\de\xi_t = \de W_t + \frac{c}{\xi_t}\de t,\quad
t<T_0,
\]
where $T_0$ is the first hitting time of 0.
If $\tmpdelta\ge2$ (i.e., $c\ge1/2$), then the Bessel process with
probability~1 does not hit 0.

If $0<\tmpdelta<2$ (\xch{i.e.,}{i.e.} $-1/2<c<1/2$), then with probability~1 the
Bessel process hits 0 but spends zero time at 0.
In particular, if $\tmpdelta=1$ (i.e., $ c=0$), then the Bessel process
is a reflecting Brownian motion.

If $\tmpdelta=0$ (i.e., $c=-1/2$), then with probability~1 the process
attains 0 and remains there forever.

The scale function of the Bessel process $B_c$ equals
\begin{equation}
\label{eq_scale1} \psi_c(x)= %
\begin{cases}
-x^{-2c+1} &\text{if }\ c>1/2,\\
\ln x &\text{if }\ c=1/2,\\
x^{-2c+1} &\text{if }\ c<1/2,
\end{cases} %
\end{equation}
that is,
\[
P_x(T_a<T_b)=\frac{\psi_c(b)-\psi_c(x)}{\psi_c(b)-\psi_c(a)}\quad \text{
for any }\ 0<a<x<b,
\]
where $T_y=\inf\{t\ge0 : \ B_c(t)=y\}$.

The transition density for $c>-1/2$, $x, y>0$, and $t>0$ equals% (see
%\cite[XI]{RevuzYor1999})
%
\[
p_t^c(x,y) = t^{-1}(y/x)^{\tmpnu} y \exp
\bigl(-\bigl(x^2+y^2\bigr)/2t\bigr) I_\tmpnu(xy/t),
\]
where $I_\tmpnu$ is a Bessel function of index
$\tmpnu= c-1/2$.

Let $c\in(-1/2,1/2)$,
and let
$p_t^{0,c}(x,y)$
be the transition density of the Bessel process $B_c$ killed at 0.

Set
\begin{align*}
p_t^{\mathit{skew}}(x,y)& = p_t^{0,c}\big(|x|,|y|\big)
\cdot\1_{xy>0}
\\
&\quad + \frac{1+\gamma\sign y}{2} \bigl(p^c_t\big(|x|,|y|\big)-p_t^{0,c}\big(|x|,|y|\big)
\bigr),\quad x,y\in\R.
\end{align*}
It is easy to
verify
that this function satisfies
the Chapman--Kolmogorov equation,
is nonnegative,
and
$\int_{\R} p_t^{\mathit{skew}}(x,y) \de y = 1,\ x\in\R$.

\begin{definition}
A time-homogeneous Markov process with the transition density
$p_t^{\mathit{skew}}$
is called \emph{the skew Bessel process} $B_{c,\gamma}^{\mathit{skew}}$
with parameters~$c$ and~$\gamma\in[-1,1]$.
\end{definition}

\begin{remark}
We do not consider the skew Bessel process for $c\ge1/2$ because
$B_c(x_0,\cdot)$ does not hit 0 if $x_0\neq0$.
\end{remark}

\begin{remark}\label{RemSk_Bess}
The skew Bessel process
$B^{\mathit{skew}}$ can be obtained from a nonnegative Bessel process
by pointing its excursions
up with probability $p=\tfrac{1+\gamma}{2}$
and down with probability $q=\tfrac{1-\gamma}{2}$, similarly to the
case of a skew Brownian motion; see arguments in \cite
{BarlowPitmanYor1989}, Section~2.

Thus, the scale function of the skew Bessel process equals
\begin{equation}
\label{eq_scale2} \psi_{\mathit{skew}}(x) = (q\1_{x\ge0}-p\1_{x< 0})
|x|^{-2c+1} ,\quad x\in\R.
\end{equation}
For other properties of the skew Bessel process, we refer to \cite{Blei2012}.
\end{remark}

\begin{remark}
If $x_0>0$ and $p=1$ (i.e., $\gamma=1$), then $B_{c,\gamma
}^{\mathit{skew}}(x_0,\cdot)$ is %a
a (nonnegative)
Bessel process $B_{c}(x_0,\cdot)$ with parameter $c$:
$B_{c,1}^{\mathit{skew}}(x_0,\cdot) \stackrel{d}{=} B_{c}(x_0,\cdot)$.

Also, the absolute value of the skew Bessel process
$|B_{c,\gamma}^{\mathit{skew}}|$ is %a
a (nonnegative) Bessel process %on $[0,\infty)$
$B_{c}(x_0,\cdot)$ with parameter $c$:
$|B_{c,\gamma}^{\mathit{skew}}(x_0,\cdot)| \stackrel{d}{=} B_{c}(x_0,\cdot)$.

If $c=0$, then $B_{c,\gamma}^{\mathit{skew}}$ is a skew Brownian motion:
$B_{0,\gamma}^{\mathit{skew}}(\cdot) \stackrel{d}{=} W_{\gamma}(\cdot)$.
\end{remark}

%\pagebreak
\section{Main results}

Let
\[
a(x) = \wta(x) + \frac{\bar{c}(x)}{x}, \quad x\in\R,
\]
where $\wta\in L_1(\R)$ and
\[
\bar{c}(x) = c_+\cdot\1_{x>1} + c_-\cdot\1_{x<-1}, \quad x\in\R.
\]

Let $\xn(t), t\ge0$, be the solution of the SDE
\[
\begin{cases}
\de\xn(t) = n a\big(n\xn(t)\big)\de t + \de W(t) \\
\qquad\qquad\ =  \bigg(n \wta\big(n\xn(t)\big)+\dfrac{\bar{c}\big(n\xn(t)\big)}{\xn
(t)} \bigg)\de t+ \de W(t), \quad t\ge0,\\
\xn(0) = x_0.
\end{cases} %
\]
The existence and uniqueness of a strong solution of this SDE follows
from \cite[Thm.~4.53]{EngelbertSchmidt1991}.

%\pagebreak
%
\begin{theorem}\label{thm:main-theorem}
If $c_+
\text{ and }
c_- > -1/2$, then the sequence of processes
$\{\xn\}$
converges weakly
to a limit process $\xn[\infty]$. In particular:
\begin{enumerate}%[leftmargin=\parindent]
\renewcommand{\theenumi}{A\arabic{enumi}}
\renewcommand{\theenumii}{(\alph{enumii})}
\renewcommand{\labelenumii}{\theenumii}

\item\label{case:positive}

If
\begin{enumerate}%[label=\alph*),ref=\arabic{enumi}\alph*]
\item[\rm(a)]\label{case:positive-a} $x_0>0$ and $c_+\ge1/2$, or
\item[\rm(b)]\label{case:positive-b} $x_0\ge0$ and $c_-<c_+<1/2$, or
\item[\rm(c)]\label{case:positive-c} $x_0=0$ and $c_-<1/2\le c_+$,
\end{enumerate}
then
\[
\xn[\infty](t) = B^+_{c_+}(x_0,t),\quad t\ge0.
\]

\item\label{case:negative}

Similarly, if
\begin{enumerate}%[label=\alph*),ref=\arabic{enumi}\alph*]
\item[\rm(a)]\label{case:negative-a} $x_0<0$ and $c_-\ge1/2$, or
\item[\rm(b)]\label{case:negative-b} $x_0\le0$ and $c_+<c_-<1/2$, or
\item[\rm(c)]\label{case:negative-c} $x_0=0$ and $c_+<1/2\le c_-$,
\end{enumerate}
then
\[
\xn[\infty](t) = B^-_{c_-}(x_0,t),\quad t\ge0.
\]

\item\label{case:positive1}
If
%\begin{enumerate}%[label=\alph*),ref=\arabic{enumi}\alph*]
%\item\label{case:positive-d}
$x_0<0$, $c_-<1/2$, and $c_-<c_+$,
%\end{enumerate}
then the limiting process evolves as~$B_{c_-}^-$ until hitting 0 and
then proceeds as $B^+_{c_+}$
indefinitely, that is,
\[
\xn[\infty](t) = B^-_{c_-}(x_0,t) \cdot\1_{t\le\tau} +
B^+_{c_+}(0,t-\tau) \cdot\1_{t>\tau},\quad t\ge0,
\]
where $\tau= \inf\{t\colon\xn[\infty](t)\ge0\}$
and
$B^\pm_{c_\pm}$
are independent \textup{(}positive and negative\textup{)}
Bessel processes.

\item\label{case:negative1}

Similarly, if
%\begin{enumerate}%[label=\alph*),ref=\arabic{enumi}\alph*]
$x_0>0$, $c_+<1/2$, and $c_+<c_-$,
%\end{enumerate}
then
\[
\xn[\infty](t) = B^+_{c_+}(x_0,t) \cdot\1_{t\le\tau} +
B^-_{c_-}(0,t-\tau) \cdot\1_{t>\tau},\quad t\ge0,
\]
where $\tau= \inf\{t\colon\xn[\infty](t)\le0\}$.

\item\label{case:skew}

If $c_+=c_- =: c <1/2$, then, for any $x_0$,
\[
\xn[\infty](t) = B_{c,\gamma}^{\mathit{skew}} (x_0,t),\quad t\ge0,
\]
where
$
\gamma= \tanh(\int_{-\infty}^{+\infty} \wta(z) \de z)
{}=
\frac{1-\exp\{-2\int_{-\infty}^{+\infty} \wta(z) \de z\}}%
{1+\exp\{-2\int_{-\infty}^{+\infty} \wta(z) \de z\}}
$.

\item\label{case:random}

Finally, if $x_0=0$,
$c_+\ge1/2$, and $c_-\ge1/2$,
then the distribution of the limit process~$\xn[\infty]$ equals
\[
p\cdot\Law_{B^+_{c_+}} + (1-p)\cdot\Law_{B^-_{c_-}},
\]
where
\begin{equation}
\label{eq_p} p = \frac{
\int_{0}^{\infty}
A(-y) (y\mx1)^{-2c_-} \de y
}{
\int_{0}^{\infty}  (
A(-y) (y\mx1)^{-2c_-} +
A(y) (y\mx1)^{-2c_+}
 ) \de y},
\end{equation}
$A(y) = \exp\{-2\int_{0}^{y}\wta(z)\de z\}$,
and $\Law_{ B^\pm_{c_\pm}}$ are the distributions of positive and
negative Bessel processes $ B^\pm_{c_\pm}(0,\cdot)$ starting from $0$.
\end{enumerate}
\end{theorem}

% \begin{remark}

% Notice that $\tau=0$ in cases \ref{case:positive-a}--
%\ref{case:positive-c}.
% Thus %the limit process
% $\xn[\infty](t) = B^+_{c_+}(x_0,t)%,\ t\ge0
% $
% is a non-negative Bessel process.
% We have $\tau= \inf\{t\geq
% 0\colon B^-_{c_-}(x_0,t)=0\}$ in the case~\ref{case:positive1}.

% Similarly,
% $\xn[\infty](t) = B^-_{c_-}(x_0,t)%,\ t\ge0
% $
% is a non-positive Bessel process in cases \ref{case:negative-a}--
%\ref{case:negative-c}.
% \end{remark}

\begin{remark}
Some results of the theorem follow from \cite{KulinichKaskun1997}.
However, we apply here the general approach applicable to all cases
simultaneously. Condition~\eqref{eqKulKas} is somewhat weaker than
$\tilde a\in L_1(\R)$. However, we do not assume that $\sup_x|x\tilde
a(x)|<\infty$, contrary to the paper
\cite{KulinichKaskun1997}.
\end{remark}

%\pagebreak
\section{Proof}

It follows from \cite[Section~3]{LeGall1983} or \cite
[Section~3.7]{Makhno2012} that if \ref{case:positive-a} is satisfied,
then, for any $\alpha>0$, we have the convergence
\[
X_n\bigl(\cdot\wedge\tau^{n,\alpha}\bigr)\quad\Rightarrow\quad
B^+_{c_+}\bigl(x_0,\cdot\wedge \tau^{0,\alpha}\bigr),\quad n
\to\infty,
\]
where
$ \tau^{n,\alpha}=\inf \{t\ge0 \colon
X_n(t) \le\alpha \}$ and $\tau^{0,\alpha}=\inf \{t\ge0 \colon
B^+_{c_+}(x_0,t) \le\alpha \}$. Since the process
$B^+_{c_+}(x_0,\cdot)$ does not hit 0, this yields the proof. Case~\ref
{case:negative-a} is considered similarly.

To prove all other items of Theorem~\ref{thm:main-theorem}, we use the
method proposed in~\cite{PilipenkoPrykhodko2015}.

Let $\{\xin,n\ge0\}$ be a sequence of continuous homogeneous strong
Markov processes.
For $\alpha>0$, set
\[
\tau^{n, \alpha} := \inf \bigl\{t\ge0 \colon \big|\xin(t)\big| \le\alpha \bigr\},
\qquad \sigma^{n, \alpha} := \inf \bigl\{t\ge0 \colon \big|\xin(t)\big| \ge\alpha \bigr
\}.
\]

We denote by $\xin_{x_0}$ a process that has the distribution of $\xin$
conditioned by\linebreak$\xin(0)=x_0$.

The next statement is a particular %or:
%special
case of Theorem 2 of~\cite{PilipenkoPrykhodko2015}.

\begin{lemma}
Assume that the sequence $\{\xin,n\ge0\}$ satisfies the following
conditions\textup{:}
\begin{align}
\label{cond:nach_usl} &\xin(0) \quad\Rightarrow\quad\xin[0](0);\\
&\forall T>0\ \forall\varepsilon>0\ \exists\delta>0\ \exists n_0\
\forall n\ge n_0\notag\\
&\quad\label{cond:modulus} \Pb \Bigl(\sup_{\substack{|s-t|<\delta, \\ s,t\in[0,T]}} \big|\xin(t) - \xin (s)\big| \ge
\varepsilon \Bigr) \le\varepsilon;\\
\label{cond:time-epsilon}& \forall T>0\qquad \lim_{\varepsilon\to0+} \sup
_n \E\int_{0}^{T}
\1_{|\xin(t)|\le
\varepsilon} \de t = 0;\\
\label{cond:time-zero}& \int_{0}^{\infty} \1_{\xin[0](t)=0} \de
t = 0\quad \almsur%.
\end{align}
Assume that, for any $\alpha>0$, $x_0\in\mbR$, and any sequence $\{x_n\}
$ such that $\lim_{n\to\infty} x_n = x_0$, we have
\begin{align}
\label{cond:stopped-tau-conv} \bigl(\xin_{x_n}\bigl(\cdot\wedge\tau^{n, \alpha}
\bigr), \tau^{n, \alpha} \bigr) &\quad\Rightarrow{}\quad \bigl(\xin[0]_{x_0}
\bigl(\cdot\wedge\tau^{0, \alpha}\bigr), \tau^{0, \alpha} \bigr),\quad n\to
\infty;\\% \ \text{ for any }\ \ x_0\neq0;
\label{cond:sigma-conv} \xin_{x_n}\bigl(\sigma^{n, \alpha} \bigr)
&\quad\Rightarrow{}\quad \xin[0]_{x_0}\bigl(\sigma^{0, \alpha} \bigr), \quad n\to
\infty. % \ \text{ for any }\
%\ x_0.
\end{align}
Then $\xin\Rightarrow\xin[0]$ in $C([0,\infty))$ as $n\to\infty$.
\end{lemma}
We apply this lemma for $\xin=\xn,n\ge1$, and $\xin[0]=\xn[\infty]$ in
cases \ref{case:positive}--\ref{case:skew} of the theorem. Case \ref
{case:random} will be considered separately.
\begin{remark}
Condition~\eqref{cond:sigma-conv} is the only condition that is not
true in case \ref{case:random}. It fails if $x_0=0$. Indeed, for any
$x>0$, the process $B^+_{c_+}( x, \cdot)$ does not hit~0. So, we may
select a sequence $\{x_n\}\subset(0,\infty)$ that converges to 0
sufficiently slowly and such that, given $X_n(0)=x_n$, we have
$X_n(\cdot)\Rightarrow B_+(0,\cdot)$ and $\lim_{n\to\infty}P(\exists
t\geq0\colon \ X_n(t)=0)=0$.
The concrete selection of $\{x_k\}$ can be done using formulas~\eqref
{eq_sc767} and \eqref{eq_sc777}. Since $B_+(0,\sigma^{0, \alpha})=\alpha
$ a.s., we get $X_n(\sigma^{n, \alpha})\Rightarrow\alpha$. However, if
all $x_n$ were negative, then the limit might be $-\alpha$.
\end{remark}

Conditions~\eqref{cond:nach_usl} and \eqref{cond:time-zero} are obvious.

The convergence
\begin{equation}
\label{cond:stopped-tau-conv1} \forall\alpha>0\ \ \ \xin_{x_n}\bigl(\cdot\wedge
\tau^{n, \alpha}\bigr) \quad\Rightarrow \quad\xin[0]_{x_0}\bigl(\cdot\wedge
\tau^{0, \alpha}\bigr),\quad n\to\infty,
\end{equation}
follows from \cite[Section~3]{LeGall1983} or \cite[Section~3.7]{Makhno2012}.
Since
\[
P \bigl(\forall\ve>0\ \exists t\in\bigl(\tau^{0, \alpha},
\tau^{0, \alpha
}+\ve\bigr) \colon\ \big|\xin[0]_{x_0}(t)\big|<\alpha\ | \
\tau^{0, \alpha}<\infty \bigr)=1,
\]
convergence \eqref{cond:stopped-tau-conv1} yields the convergence of
pairs~\eqref{cond:stopped-tau-conv}.
% \begin{remark}
% The convergence of pairs in~\eqref{cond:stopped-tau-conv} is the
%extra assumption of the Proposition, it may be replaced by the
%convergence of the first coordinates only. However, there %was no
%reason
% were no reasons
% to study details of proofs in \cite{PilipenkoPrykhodko2015} that are
%applied to our problem.
% \end{remark}

Let us check condition~\eqref{cond:modulus}.
Set
\begin{align*}
A(y) &= \exp\Biggl\{-2\int_{0}^{y} \wta(z) \de z
\Biggr\},\\
A_n(y) &= \exp\Biggl\{-2\int_{0}^{y}
n\wta(nz) \de z\Biggr\} = A(ny),\quad y\in\R,\\
\varPhi_n(x)&=\int_0^x
A_n(y)\de y,\quad x\in\R.
\end{align*}
Observe that $\varPhi_n:\R\to\R$ is a bijection, $\varPhi_n(0)=0$, and
\[
\exists L>0\ \ \forall n\ \ \forall x, y\in\mbR\qquad L^{-1}|x-y|\le \big|
\varPhi_n(x)-\varPhi_n(y)\big|\le L|x-y|.
\]
It\^o's formula yields
\[
\de\varPhi_n\bigl(\xn(t)\bigr)= A\bigl(n\xn(t)\bigr) \biggl(
\frac{\bar c(n\xn(t))}{\xn(t)} \de t + \de W(t) \biggr).
\]
So
\begin{align*}
\bigl\llvert \xn(t)-\xn(s)\bigr\rrvert &\le L\bigl\llvert \varPhi_n
\bigl(\xn(t)\bigr)-\varPhi_n\bigl(\xn(s)\bigr)\bigr\rrvert \\
&\le L\Biggl\llvert \int_s^t A\bigl(n\xn(z)
\bigr) \frac{\bar c(n\xn(z))}{\xn(z)} \de z\Biggr\rrvert + L\Biggl\llvert \int
_s^t A\bigl(n\xn(z)\bigr) \de W(z)\Biggr\rrvert
.
\end{align*}

Let $|s-t|<\delta$, and let $\varDelta>0$ be fixed.
Denote $f_n(t) := \int_0^t A(n\xn(z)) \de W(z)$.

a) Assume that $|\xn(z)|>\varDelta, z\in[s,t]$.
Then
\[
\Biggl\llvert \int_s^t A\bigl(n\xn(z)\bigr)
\frac{\bar c(n\xn(z))}{\xn(z)} \de z\Biggr\rrvert \le C\delta/\varDelta,
\]
where $C=\|A\|_\infty\max(|c_-|, |c_+|)<\infty$. Hence, we have the estimate
\[
\big|\xn(t)-\xn(s)\big|\le LC\delta/\varDelta+L\omega_{f_n}(\delta),
\]
where $\omega_f(\delta)=\sup_{|s-t|<\delta, \ s,t\in[0,T]}|f(t)-f(s)|$
is the modulus of continuity.\vadjust{\eject}

b) Assume that $|\xn(z_0)|\le\varDelta$ for some $z_0\in[s,t]$.

Denote $\tau:=\inf\{z\ge s \colon |\xn(z)|\le\varDelta\}$ and $\sigma:=\sup\{
z\le t \colon |\xn(z)|\le\varDelta\}$. Then
\begingroup\abovedisplayskip=6pt
\belowdisplayskip=6pt
\begin{align*}
\big|\xn(t) - \xn(s)\big| &\le\big|\xn(s) - \xn(\tau)\big| + \big|\xn(\sigma) - \xn(t)\big| + 2\varDelta\\%[-2pt]
&\le2LC\delta/\varDelta+2L\omega_{f_n}(\delta)+ 2\varDelta.
\end{align*}

Thus, in any case, we have the following estimate of the modulus of continuity:
\[
\omega_{\xn}(\delta) \le2LC\delta/\varDelta+2L\omega_{f_n}(
\delta)+ 2\varDelta.
\]

Let $\varDelta\le\varepsilon/6$. Then, for
$\delta\le\frac{\varepsilon\varDelta}{6LC}$, we have
\begin{align*}
\sup_n \Pb\bigl(\omega_{\xn}(\delta)\ge\varepsilon
\bigr)& \le \sup_n\Pb\bigl(\ve/3+2L\omega_{f_n}(
\delta)+\ve/3\ge\varepsilon\bigr)\\%[-2pt]
&= \sup_n\Pb\bigl(\omega_{f_n}(\delta)\ge
\varepsilon/6L\bigr) \to 0,\quad \delta\to0+.
\end{align*}
The last convergence follows from the fact that the sequence of
distributions of
$\{f_n(\cdot) = \int_0^. A(n\xn(z)) \de W(z)\}_{n\ge1}$
in the space of continuous functions is weakly relatively compact
because the function $A$ is bounded.

Let us prove \eqref{cond:sigma-conv} in cases \ref{case:positive}--\ref
{case:skew}.
\begin{remark}
The proof below yields that condition~\eqref{cond:sigma-conv} is true
if $x_n=0$ for all $n\ge0$.
\end{remark}
Let $|x|<\alpha$. It is easy to see that
$P_x(\sigma^{n,\alpha}<\infty)=1, \ n\in\mbN\cup\{\infty\}$. Since the
process $\xn$ is continuous, we have
$|\xn(\sigma^{n,\alpha})| = \alpha$ a.s.% (the process $X_\infty$ is
%defined in the formulation of Theorem \ref{thm:main-theorem}).

By $p^n_x=P_x(\xn(\sigma^{n,\alpha}) = \alpha), \ n\in\mbN\cup\{\infty
\}$, we denote the probability
to reach~$\alpha$ before reaching~$-\alpha$ when starting from~$x$.

% Recall that for any regular diffusion $\xi$ and all $a\le x\le b$ the
%following relation holds
% $$
% P_x(T_a>T_b)=\frac{\vf(x)-\vf(a)}{\vf(b)-\vf(a)},
% $$
% where $T_y=\inf\{t\ge0 : \xi(t)=y\},$ $\vf$ is the scale function of
%diffusion $\xi$ (see \cite[Ch. VII, \S3]{RevuzYor1999}).

% The scale function of the Bessel process $B_c(t),\ t\ge0,$ equals
%(see \eqref{eq_scale1}, \eqref{eq_scale2})
% $$
% \psi_c(x) = \begin{cases}
% -x^{1-2c } &\text{if } c>1/2,\\
% \ln x &\text{if } c=1/2,\\
% x^{1-2c} &\text{if } c<1/2
% \end{cases}
% $$
% for $x>0.$

Using formulas \eqref{eq_scale1} and \eqref{eq_scale2} for the scale of
a Bessel process and a skew Bessel
process, it
% The scale function of the skew Bessel process equals (see Remark
%\ref{RemSk_Bess})
% $$
% q\psi_c(x)\1_{x\ge0}-p\psi_c(-x)\1_{x< 0}
% =
% (q\1_{x\ge0}-p\1_{x< 0}) \psi_c(|x|)
% ,\ x\in\R.
% $$
is easy to check that
\begin{equation}
\label{eq_sc757} p^\infty_x = %
\begin{cases}
\1_{x\ge0} -  \bigl(1-\frac{\psi_{c_-}(-x)}{\psi_{c_-}(\alpha)} \bigr)\1
_{x< 0}
&\text{\xch{in cases}{in e cases} \ref{case:positive}, \ref{case:positive1}},
\\[4pt]
\frac{\psi_{c_+}(x)}{\psi_{c_+}(\alpha)}\cdot\1_{x>0}
&\text{in cases \ref{case:negative}, \ref{case:negative1}},
\\[6pt]
\frac{\psi_c(|x|)}{\psi_c(\alpha)}(q\1_{x\ge0}-p\1_{x<0})+p
&\text{in case \ref{case:skew}},
\end{cases} %
\end{equation}
where $\psi_c$ is given in \eqref{eq_scale1}.

For $n\in\mbN$, we have (see \cite[Section~15]{GikhmanSkorokhod1968} and
\cite{RevuzYor1999})
\begin{equation}
\label{eq_sc767} p^n_x = \frac{\varphi_n(x) - \varphi_n(-\alpha)}{\varphi_n(\alpha) -
\varphi_n(-\alpha)},
\end{equation}
where
\begin{align}
\label{eq_sc777} \varphi_n(x) &= \int_{0}^{x}
\exp\Biggl\{-2\int_{0}^{y} a_n(z) \de
z\Biggr\} \de y= \int_{0}^{x} \exp\Biggl\{-2\int
_{0}^{y} n a(nz) \de z\Biggr\} \de y\notag\\%[-2pt]
& =
% $$
% $$
% =
\dfrac{1}{n} \int_{0}^{nx}
\exp\Biggl\{-2\int_{0}^{y} a(z) \de z\Biggr\} \de
y = \dfrac{1}{n} \vf(nx),\\%[-2pt]
\vf(x) &:= \int_{0}^{x} \exp\Biggl\{-2\int
_{0}^{y} a(z) \de z\Biggr\} \de y.\notag
\end{align}
\endgroup
The function $\vf$ is increasing. It follows from the definition of $a$
that $\vf$ is bounded from above (below) if{}f $ c_+>1/2$ ($c_->1/2$).
The function $\vf$ has the following asymptotic behavior:
\begin{equation}
\label{eq_sc787} \varphi(x) \sim %
\begin{cases}
\pm A(\pm\infty) \frac{|x|^{1-2c_\pm}}{1-2c_\pm}& \text{ if }\ c_\pm
<1/2,\\
\pm A(\pm\infty) \ln|x|& \text{ if }\ c_\pm=1/2,
\end{cases} %
\quad x\to\pm\infty,
\end{equation}
where
\[
A(y) = \exp\Biggl\{-2\int_{0}^{y} \wta(z) \de z
\Biggr\},\quad y\in\R,
\]
and
\begin{equation}
\label{eq_sc788} \lim_{x\to\pm\infty} \varphi(x) = \vf(\pm\infty)\in\mbR
\quad\text{ if }\ c_\pm>1/2.
\end{equation}

Condition \eqref{cond:sigma-conv} follows from \eqref{eq_sc757}, \eqref
{eq_sc767}, \eqref{eq_sc777}, \eqref{eq_sc787}, \eqref{eq_sc788} in
cases \ref{case:positive}--\ref{case:skew} (and in case \ref
{case:random} if $x_n=0,\ n\ge0$).

Set $\tau_n=\inf\{t\ge0 \colon |\xn(t)|\ge1\}$.
\begin{lemma} Assume that
\begin{equation}
\label{eq_770} \lim_{\varepsilon\to0+} \sup_{|x|\le1}\sup
_n \E_x \int_{0}^{\tau_n}
\1 _{|\xn(t)|\le\varepsilon} \de t = 0.
\end{equation}
Then \eqref{cond:time-epsilon} is satisfied, that is,
\[
\forall T>0\quad \lim_{\varepsilon\to0+} \sup_n \E\int
_{0}^{T} \1_{|\xn
(t)|\le\varepsilon} \de t = 0.
\]
\end{lemma}
\begin{proof}
Introduce the notations
\begin{align*}
S_{n,\pm}^0&:=0, \qquad T_{n,\pm}^{k}:=\inf
\bigl\{ t\geq S_{n,\pm}^{k-1} \colon\ X_n(t) =\pm1\bigr
\},\\
S_{n,\pm}^{k}&:=\inf\bigl\{ t\geq T_{n,\pm}^k
\colon\ X_n(t) = \pm\ve\bigr\},\\
\tilde T_{n,\pm}^{k}&:=\inf\bigl\{ t\geq S_{n,\pm}^k
\colon\ \big|X_n(t)\big| = 1 \bigr\},\\
\beta_{n,\pm}^k&:= \int^{\tilde T_{n,\pm}^{k}}_{S_{n,\pm}^{k}}
\1_{|\xn
(t)|\le\varepsilon} \de t , \qquad \alpha_{n,\pm}^k:=
S_{n,\pm}^{k}-T_{n,\pm}^{k}, \quad \xch{k\geqslant1}{k\geq1}.
\end{align*}
Then
\begin{align*}
&\int_{0}^{T} \1_{|\xn(t)|\le\varepsilon} \de t\\
&\quad\le \int_{0}^{\tau_n} \1_{|\xn(t)|\le\varepsilon} \de t+ \sum
_{k}\bigl(\beta_{n,+}^1+
\cdots+\beta_{n,+}^k\bigr)\1_{\alpha_{n,+}^1<T,\dots
,\alpha_{n,+}^{k}<T,
\alpha_{n,+}^{k+1}\xch{\geqslant}{\geq} T} {}\\
&\qquad + \sum_{k}\bigl(\beta_{n,-}^1+
\cdots+\beta_{n,-}^k\bigr)\1_{\alpha_{n,-}^1<T,\dots
,\alpha_{n,-}^{k}<T,
\alpha_{n,-}^{k+1}\xch{\geqslant}{\geq} T}\,.
\end{align*}
It follows from the strong Markov property
that
\begin{align*}
&\sum_k\E\bigl(\beta_{n,+}^1+
\cdots+\beta_{n,+}^k\bigr)\1_{\alpha_{n,+}^1<T,\dots
,\alpha_{n,+}^{k}<T,
\alpha_{n,+}^{k+1}\xch{\geqslant}{\geq} T}\\
&\quad= \sum_k k \E_\ve\int
_{0}^{\tau_n} \1_{|\xn(t)|\le\varepsilon} \de t \,
(1-p_{n,+})^k p_{n,+}= (p_{n,+})^{-1}
\E_\ve\int_{0}^{\tau_n} \1_{|\xn
(t)|\le\varepsilon}
\de t ,
\end{align*}
where $p_{n,+}= \Pb_1(S_{n,+}^{1}\xch{\geqslant}{\geq} T)$.

Considering the last term similarly, we get the inequality
\[
\E\int_{0}^{T} \1_{|\xn(t)|\le\varepsilon} \de t \le
\bigl(1+ (p_{n,+})^{-1}+ (p_{n,-})^{-1}
\bigr) \sup_{|x|\le1}\sup_n \E _x
\int_{0}^{\tau_n} \1_{|\xn(t)|\le\varepsilon} \de t.
\]
It is not difficult to see that $\sup_n(p_{n,\pm})^{-1}<\infty$. The
lemma is proved.
\end{proof}

Let us verify \eqref{eq_770}.
It is known \cite[Chap.~4.3]{Knight1981} that
\[
u_{n,\ve}(x) := \E_x \int_{0}^{\tau_n}
\1_{|\xn(t)|\le\varepsilon} \de t
\]
is of the form
\begin{equation}
\label{eq_time746} u_{n,\ve}(x)=\int_{-1}^1G_n(x,y)
\1_{|y|\le\ve}\,m_n(d y),
\end{equation}
where Green's function $G_n$ equals
\[
G_n(x,y)= %
\begin{cases}
\frac{(\vf_n(x)-\vf_n(-1))(\vf_n(1)-\vf_n(y))}{\vf_n(1)-\vf_n(-1)},
& x\le y,\\
G_n(y,x),\
& x\ge y,
\end{cases} %
\]
with $\vf_n$ given by formula \eqref{eq_sc777}, and
\[
m_n(d x)= \exp\Biggl\{2\int_{0}^{x}
a_n(z) \de z\Biggr\} \de x.
\]
The function $u_{n,\ve}(x)$ is a generalized solution (because $a_n$
may be discontinuous) of the equation
\[
1/2 \, u_{n,\ve}''(x)+ a_n(x)
u'_{n,\ve}(x)=-\1_{|x|\le\ve}(x),\quad |x|\le1,
\]
with boundary conditions $u_{n,\ve}(\pm1)=0$.

A direct verification of the condition
$\lim_{\varepsilon\to0+} \sup_{|x|\le1}\sup_n u_{n,\ve}(x) = 0$
is possible but cumbersome. We prove the corresponding convergence
using the~comparison theorem.
We consider only the case where $a_n$ satisfies the Lipschitz
condition. The general case follows by approximation.

It follows from the It\^o--Tanaka formula that
\begin{align*}
\de\big|\xn(t)\big|&= \sign\bigl(\xn(t)\bigr)\, a_n\bigl(\xn(t)\bigr) \de t +
\sign\bigl(\xn(t)\bigr)\de W(t) +\de l_n(t)\\
&= \sign\bigl(\xn(t)\bigr)\, a_n\bigl(\xn(t)\bigr) \de t +\de
W_n(t) +\de l_n(t),
\end{align*}
where $W_n$ is a new Wiener process, and $l_n$ is the local time of $\xn
$ at zero.

Let $-1/2 < c < \min(c_-, c_+,0)$.
It follows from the arguments of \cite{PieraMazumdar2008} on comparison
of reflecting SDEs that
$|\xn(t)|\ge Y_n(t),\ t\ge0$, where
$Y_n $ satisfies the following SDE with reflection at zero:
\[
\de Y_n(t)=\bar a_n\bigl(Y_n(t)\bigr)\de t+
\de W_n(t) +\de\tilde l_n(t).
\]
Here $W_n(t)=\int_0^t\sign(\xn(s))\de W(s)$ is a Wiener process,
$\tilde l_n$ is the local time of $Y_n$ at 0, $\bar a_n(x)=n\bar
a(nx),\ \bar a(x)= - (|a(x)|+|a(-x)|)-\frac{ c}{x}\1_{|x|>1}+r(x)$, and
$r$ is any nonpositive function such that $\bar a$ satisfies Lipschitz
condition. We will also assume that $\int_\R|r(x)|\de x\le\int_\R
|b(x)|\de x$. The Lipschitz property is used only for application of
comparison theorem.

To prove \eqref{eq_770}, it suffices to verify that
\[
\lim_{\varepsilon\to0} \sup_{x\in[0,1]}\sup
_n \bar u_{n,\ve}(x):= \lim_{\varepsilon\to0}
\sup_{x\in[0,1]}\sup_n \E_x \int
_{0}^{\bar\tau_n} \1 _{ Y_n(s)\in[0,\varepsilon]} \de s = 0,
\]
where
$
\bar\tau_n
$
is the entry time of $Y_n$ into $[1, \infty)$.

It is known \cite{Knight1981} that
\[
\bar u_{n,\ve}(x)=2\int_x^1\exp\Biggl
\{-2\int_1^y\bar a_n(z)\de z\Biggr\}
\int_0^y \1_{[0,\ve]}(z) \exp\Biggl\{2
\int_1^y\bar a_n(z)\de z\Biggr\} \de
y
\]
is a (generalized) solution of the equation
\[
1/2\, \bar u_{n,\ve}''(x)+\bar
a_n(x)\bar u'_{n,\ve}(x)=-\1_{[0,\ve]},\quad
x\in[0,1],
\]
with boundary conditions
$u_{n,\ve}'(0)=0,\ u_{n,\ve}(1)=0$.
So
\begin{align}
&\sup_{x\in[0,1]}\sup_n \bar
u_{n,\ve}(x)\notag\\
&\quad = \bar u_{n,\ve}(0)\notag\\
&\quad= 2\int_0^1\exp\Biggl\{-2\int
_1^y\bar a_n(z)\de z\Biggr\} \int
_0^y \1_{[0,\ve]}(z) \exp\Biggl\{2\int
_1^y\bar a_n(z)\de z\Biggr\} \de y\notag\\
&\quad\label{eq_838} \le K\int_0^1\exp\Biggl\{\int
_0^y\bar y^{-2c}\de z\Biggr\} \int
_0^y \1_{[0,\ve]}(z)\, y^{2c}\de
y,
\end{align}
where $K$ is a constant that depends only on $\int_\R|b(x)|\de x$ and
$c$ (and is independent of~$n$). By our choice, $c\in(-1/2,0)$, so
the right-hand side of \eqref{eq_838} tends to~0 as $\ve\to0+$ by the
Lebesgue dominated convergence theorem.

The theorem is proved in cases \ref{case:positive}--\ref{case:skew}.

Consider case \ref{case:random}.
Note that conditions \eqref{cond:nach_usl}--\eqref
{cond:stopped-tau-conv} are satisfied for $\xin=\xn$, $n\ge1$, and $\xin
[0]=\xn[\infty]$, where $\xn[\infty]$
is given in the theorem. In particular, this implies that the sequence
of distributions of stochastic processes
$\{\xn\}$ in the space of continuous functions is weakly relatively
compact. Choosing an arbitrary convergent subsequence, without loss of
generality, we may assume that $\{\xn\}$ itself converges weakly to a
continuous process $X$. Let $\delta>0$, and let $\sigma^{n,\delta}=\inf
\{t\ge0  \colon \xn(t)=\delta\},\ \sigma^{ \delta}=\inf\{t\ge0 \colon
X(t)=\delta\}$.
It follows from formulas for the scale function of the processes $\{\xn
\}$ that
$\lim_{n\to\infty}P(\xn(\sigma^{n,\delta})=\delta)=p,\ \lim_{n\to\infty
}P(\xn(\sigma^{n,\delta})=-\delta)=1-p$,
where $p$ is given by \eqref{eq_p}. Formulas \eqref{cond:time-epsilon}
and \eqref{cond:stopped-tau-conv} imply that the limit process exits
from the interval
$[-\delta,\delta]$ with probability~1.

Observe that,
for almost all $\delta>0$, with respect to the Lebesgue measure, the
distribution of $X_n(\sigma^{n,\delta}+\cdot)$ %on the one hand
converges weakly as $n\to\infty$ to the distribution of
$X(\delta+\cdot)$.\break
Indeed, by the Skorokhod theorem on a single probability space (see
\cite{Skorokhod1965studies}),
without\vadjust{\eject} loss of generality, we may assume that the sequence $\{X_n\}$
converges to $X$ uniformly on compact sets with probability 1. For
simplicity, we will assume that the convergence holds for all $\omega$
and that also $\sigma^{n,\delta}, \sigma^\delta<\infty$ for all $\omega
, n, \delta>0$. So we show convergence
\begin{equation}
\label{eq801} X_n\bigl(\sigma^{n,\delta}+\cdot\bigr)\to X\bigl(
\sigma^{\delta}+\cdot\bigr)
\end{equation}
if we prove that
\begin{equation}
\label{eq802} \sigma^{n,\delta}\to\sigma^\delta,\quad n\to\infty.
\end{equation}
Convergence \eqref{eq802} may fail only if $\sigma^\delta$
is a point of a local maximum of $X$. It follows from the definition
that $\sigma^\delta$ is a point of a strict local maximum of $X$ from
the left. The set of points of local maximums that are strict maximums
from the left is at most countable. This yields that, for almost all
$\omega$ and almost all $\delta>0$ with respect
to the Lebesgue measure, we have convergence \eqref{eq802} and
hence~\eqref{eq801}.

On the other hand, the distribution of $X_n(\sigma^{n,\delta}+\,\cdot)$
%%on the one hand
converges weakly as $n\to\infty$ to the distribution of the process $\1
_{\varOmega_-}B^-_{c_-}(-\delta, \cdot)+\1_{\varOmega_+}B^+_{c_+}(\delta,
\cdot)$, where $P(\varOmega_-)=1-p,\ P(\varOmega_+)=p$, and the $\sigma
$-algebra $\{\emptyset, \varOmega_-, \varOmega_+, \varOmega\}$ is independent of
$\sigma(B^\pm_{c_\pm}(\pm\delta, t) ,\break t\ge0)$.

Recall that assumptions of the theorem yield
\[
P \bigl(\exists t\ge0\colon B^\pm_{c_\pm}(\pm\delta, t)=0 \bigr)=0.
\]
It follows from \eqref{cond:time-epsilon} that
\[
P \Biggl(\int_0^\infty\1_{X(s)=0}\de s=0
\Biggr)=1.
\]
Thus, we have the almost sure convergence in $C([0,\infty))$
\[
X\bigl(\sigma^\delta+\cdot\bigr)\to X( \cdot), \quad \delta\to0.
\]
The processes $\1_{\varOmega_-}B^-_{c_-}(-\delta, \cdot)+\1_{\varOmega
_+}B^+_{c_+}(\delta, \cdot)$ converge in distribution
to
\[
\1_{\varOmega_-}B^-_{c_-}(0, \cdot)+\1_{\varOmega_+}B^+_{c_+}(0,
\cdot),
\]
where the $\sigma$-algebras $\{\emptyset, \varOmega_-, \varOmega_+, \varOmega\}$
and $\sigma(B^\pm_{c_\pm}(0, t) , t\ge0)$ are independent.

This completes the proof of %the
Theorem~\ref{thm:main-theorem}.

%% Acknowledgements %%
%%%%%%%%%%%%%%%%%%%%%%
\section*{Acknowledgments}
The authors thank the anonymous referee for
valuable comments that helped improving the presentation.

\end{document}